\documentclass[11pt,a4paper]{amsart}
\usepackage{amsmath,amsthm,amssymb,latexsym}
\usepackage{graphicx}

\newcommand{\bd}{{\mathbb{D}}}

\newcommand{\bz}{{\mathbb{Z}}}
\newcommand{\bc}{{\mathbb{C}}}

\newcommand{\cs}{\mathcal{S}}

\renewcommand{\a}{\alpha}
\renewcommand{\b}{\beta}
\renewcommand{\l}{\lambda}

\newcommand{\s}{\sigma}

\renewcommand{\d}{\delta}
\newcommand{\dd}{\Delta}
\renewcommand{\o}{\omega}

\newcommand{\g}{\gamma}

\newcommand{\ep}{\varepsilon}
\newcommand{\z}{\zeta}

\newcommand{\ka}{\kappa}

\newcommand{\bsl}{\backslash}
\newcommand{\ovl}{\overline}

\DeclareMathOperator{\diag}{\rm diag}

\DeclareMathOperator{\dist}{\rm dist}

\allowdisplaybreaks
\numberwithin{equation}{section}

\newtheorem{theorem}{Theorem}[section]

\theoremstyle{definition}

\newtheorem{remark}[theorem]{Remark}

\begin{document}

\title[Jacobi matrices]
{A remark on the discrete spectrum of non-self-adjoint Jacobi operators}
\author[L. Golinskii]{L. Golinskii}

\address{B. Verkin Institute for Low Temperature Physics and
Engineering, Ukrainian Academy of Sciences, 47 Nauky ave., Kharkiv 61103, Ukraine}
\email{golinskii@ilt.kharkov.ua}

\date{\today}


\maketitle

\begin{abstract}
We study the trace class perturbations of the whole-line, discrete Laplacian and obtain a new bound for the
perturbation determinant of the corresponding non-self-adjoint Jacobi operator. Based on this bound, we refine
the Lieb--Thirring inequality due to Hansmann--Katriel. The spectral enclosure for such operators is also discussed.
\end{abstract}

\section*{Introduction}

In the last two decades there was a splash of activity around the spectral theory of non-self-adjoint perturbations
of some classical operators of mathematical physics, such as the Laplace and Dirac operators on the whole space,
their fractional powers, and others. Recently, there has been some interest in studying certain discrete models of the above
problem. In particular, the structure of the spectrum for compact, non-self-adjoint perturbations of the free Jacobi
and the discrete Dirac operators has attracted much attention lately. Actually the problem concerns the discrete component
of the spectrum and the rate of its accumulation to the essential spectrum. Such type of results under various assumptions
on the perturbations are united under a common name {\it Lieb--Thirring inequalities}. In the case of the free whole-line Jacobi
operator, such inequalities include the distance from an eigenvalue to the whole essential spectrum $[-2,2]$, as well as the 
distance to its endpoints. For a nice account of the existing results on the Lieb--Thirring inequalities for non-self-adjoint Jacobi 
operators, the reader may consult two recent surveys \cite{haka11} and \cite[Section 5.13]{fran20}, and references therein.

The main object under consideration is a whole-line Jacobi matrix
\begin{equation}\label{jacmat}
J=J(\{a_j\},\{b_j\},\{c_j\})_{j\in\bz} =\begin{bmatrix}
\ddots & \ddots & \ddots &  & \\
 & a_{-1} & b_0 & c_0 &  & \\
 & & a_0 & b_1 & c_1 & & \\
 & &  & a_1 & b_2 & c_2 & \\
 & &  & & \ddots & \ddots & \ddots
\end{bmatrix}, 
\end{equation}
with uniformly bounded complex entries and $a_nc_n\not=0$. The spectral theory of the underlying non-self-adjoint Jacobi operator includes, 
among others, the structure of their spectra. We denote by $J_0$ the discrete Laplacian, i.e., $J_0=J(\{1\},\{0\},\{1\})$.
If $J-J_0$ is a compact operator, that is,
$$ \lim_{n\to\pm\infty} a_n-1=\lim_{n\to\pm\infty} c_n-1=\lim_{n\to\pm\infty} b_n=0, $$
the geometric image of the spectrum is utterly clear
$$ \s(J)=\s_{ess}(J_0)\cup \s_d(J)=[-2,2]\cup \s_d(J), $$
the discrete component $\s_d(J)$ is an at most countable set of points in $\bc\backslash [-2,2]$ with the only possible limit points on $[-2,2]$. 
To get some quantitative information on the rate of accumulation one has to impose further assumptions on the perturbation. Our case of
study in this note is the {\it trace class} perturbations of the discrete Laplacian
\begin{equation}\label{trclper}
J-J_0\in\cs_1 \ \Leftrightarrow \ \sum_{n=-\infty}^\infty (|1-a_n|+|b_n|+|1-c_n|)<\infty.
\end{equation}
Now the discrete spectrum is the set of isolated eigenvalues of finite algebraic multiplicity.

The currently best result which governs the behavior of the discrete spectrum is due to Hansmann--Katriel \cite[Theorem 1]{haka11}. It states that
for each $\ep\in(0,1)$ there is a constant $C(\ep)>0$ so that
\begin{equation}\label{hankat}
\sum_{\l\in\s_d(J)} \frac{\dist(\l,[-2,2])^{1+\ep}}{|\l^2-4|^{\frac12+\frac{\ep}4}}\le C(\ep)\|J-J_0\|_1.
\end{equation}
The result is known to be sharp \cite{bostam20} in the sense that \eqref{hankat} is false for $\ep=0$. Yet the question arises naturally whether it
is possible to drop at least one of the two small parameters on the left side. We answer this question affirmatively in this note.
The price we pay is a constant on the right side.

\begin{theorem}\label{mainth}
Let $J-J_0\in\cs_1$. Then for each $\ep\in(0,1)$ there is a constant $C(\ep)>0$ so that
\begin{equation}\label{mainlt}
\sum_{\l\in\s_d(J)} \frac{\dist(\l,[-2,2])}{|\l^2-4|^{{\frac{1-\ep}2}}}\le C(\ep)\dd, \quad \dd:=\sum_{n=-\infty}^\infty (|b_n|+|1-a_nc_n|).
\end{equation}
\end{theorem}

If $J$ is a discrete Schr\"odinger operator, that is, $a_n=c_n\equiv 1$, then
\begin{equation}\label{mainlt1}
\sum_{\l\in\s_d(J)} \frac{\dist(\l,[-2,2])}{|\l^2-4|^{{\frac{1-\ep}2}}}\le C(\ep)\|J-J_0\|_1.
\end{equation}

\begin{remark}
The appearance of the value $\dd$ in place of $\|J-J_0\|_1$ might seem reasonable. Indeed, given a Jacobi matrix $J$, consider a class $S(J)$ 
of Jacobi matrices
\begin{equation*}
\begin{split}
S(J) &=\{\widehat J:=T^{-1}JT, \ T=\diag(t_j)_{j\in\bz} \ {\rm is \ a \ diagonal \ isomorphism \ of}\ \ell^2(\bz)\}, \\
\widehat J &=J\bigl(\{a_jr_j\}, \{b_j\}, \{c_jr_j^{-1}\}\bigr), \quad r_n=\frac{t_n}{t_{n+1}}, \quad n\in\bz.
\end{split}
\end{equation*} 
As $\widehat J$ is similar to $J$, the equality $\s_d(\widehat J)=\s_d(J)$ holds. So the left side of \eqref{mainlt} does not alter within the class
$S(J)$. The same is true for the value $\dd$, in contrast to $\|J-J_0\|_1$.
For the class $S(J_0)$ both sides of \eqref{mainlt} vanish, whereas $\|J-J_0\|_1$, $J\in S(J_0)$, can be arbitrarily large.

Next, $|1-a_nc_n|\le |1-a_n|+|1-c_n|+|1-a_n||1-c_n|$, and so
$$ \dd\le 3\|J-J_0\|_1+\|J-J_0\|^2. $$
We see that for small perturbations the value $\dd$ has at least the same order as $\|J-J_0\|_1$.
\end{remark}

The so-called {\it perturbation determinant} 
\begin{equation*}
L(\l,J):=\det(I+(J-J_0)(J_0-\l)^{-1}),
\end{equation*}
introduced by M.G. Krein \cite{gokr67} in the late 50-th, comes in as a principal analytic tool. The main feature
of this analytic function on the resolvent set $\rho(J_0)=\ovl\bc\bsl [-2,2]$ is that the zero divisor agrees with the discrete spectrum
of the perturbed operator $J$, and moreover, the multiplicity of each zero equals the algebraic multiplicity of the corresponding eigenvalue.
So the original problem of spectral theory can be restated as the classical problem of the zero distributions of analytic functions, which goes back 
to Jensen and Blaschke.

The arguments in \cite{haka11} pursue in two steps. The first one results in a certain bound for the perturbation determinant, typical for the functions
of non-radial growth. The classes of such analytic (and subharmonic) functions in the unit disk were introduced and studied in \cite{bgk09, fago09} (for
some advances see \cite{bgk18}). The Blaschke-type conditions for the zero sets (Riesz measures) were proved therein, with an important amplification
in \cite[Theorem 4]{haka11}, better adapted for applications. The second step is just the latter result applied to the bound mentioned above.

In our approach to the problem the argument in the first step is totally different. Instead of certain operator-theoretic means and the Fourier
transform, we deal with the associated three-term recurrence relation
\begin{equation}\label{3term}
a_{k-1}u_{k-1}+b_ku_k+c_ku_{k+1}=\l(z)u_k, \qquad k\in\bz, \quad \l(z)=z+\frac1z,
\end{equation}
and its modifications. Here $\l(\cdot)$ is the Zhukovsky function which maps the unit disk onto the resolvent set $\rho(J_0)=\ovl\bc\bsl [-2,2]$.
The solution of \eqref{3term} $u=(u_k)_{k\in\bz}$ from $\ell^2(\bz)$ is exactly the eigenvector of $J$ with the eigenvalue $\l$. Next, the solutions
$u^{\pm}=(u^{\pm}_k)_{k\in\bz}$ are called the {\it Jost solutions at $\pm\infty$} if
\begin{equation}\label{jossol}
\lim_{n\to\pm\infty} z^{\mp n}u_n^{\pm}(z)=1, \qquad z\in\bd_0:=\bd\bsl\{0\}.
\end{equation}
We study the Jost solutions by reducing the difference equation \eqref{3term} to a Volterra-type discrete integral equation,  
see, e.g. \cite[Section 7.5]{te00} and \cite{eggo05}. The bounds for the Jost solutions stem from successive approximations method. 
The perturbation determinant arises as the Wronskian of the Jost solutions, so its bound is then straightforward.

Note also, that the relation
\begin{equation}\label{pede}
|L(z,J)-1| \le (4x+5x^2)e^{4x}, \qquad x:=\frac{2|z|}{|1-z^2|}(\dd^{1/2}+\dd) 
\end{equation}
for $z$ in the open unit disk $\bd:=\{|z|<1\}$ provides some information about the spectral enclosure, see \eqref{perdetto1} and Remark \ref{speenc}.

\section{Jost solutions and discrete Volterra equations}
\label{s1}

The following two companions of the main difference equation \eqref{3term} are of particular concern
\begin{equation}\label{der}
v_{k-1}(z)+b_kv_k(z)+a_kc_kv_{k+1}(z)=\Bigl(z+\frac1z\Bigr)v_k(z), \qquad k\in\bz,
\end{equation}
and
\begin{equation}\label{del}
a_{k-1}c_{k-1}w_{k-1}(z)+b_kw_k(z)+w_{k+1}(z)=\Bigl(z+\frac1z\Bigr)w_k(z), \qquad k\in\bz,
\end{equation}
$z\in\bd_0$. Put
\begin{equation}\label{equiv}
\a_n:=\prod_{j=-\infty}^{n-1} a_j, \quad \g_n:=\prod_{j=-\infty}^{n-1} c_j^{-1}, \quad n\in\bz.
\end{equation}
It is easy to see that $u=(u_k)$ is a solution of \eqref{3term} if and only if 
$$ u_k=\a_kv_k, \qquad \Bigl(u_k=\g_kw_k\Bigr), \qquad k\in\bz, $$
where $v=(v_k)$ ($w=(w_k)$) is a solution of \eqref{der} (\eqref{del}), respectively. In particular, if $u^{\pm}=(u^{\pm}_k)$ are the
Jost solutions of \eqref{3term}, then
\begin{equation}\label{trans}
\begin{split}
u_n^+ &=\prod_{j=n}^\infty a_j^{-1}\,v_n^+=\prod_{j=n}^\infty c_j\,w_n^+, \\
u_n^- &=\prod_{j=-\infty}^{n-1} c_j^{-1}\,w_n^-=\prod_{j=-\infty}^{n-1} a_j\,v_n^-,
\end{split}
\end{equation}
where $v^{\pm}=(u^{\pm}_k)$ ($w^{\pm}=(w^{\pm}_k)$) are the Jost solutions of \eqref{der} (\eqref{del}), respectively.

We are aimed at obtaining the bounds for the Jost solutions $v^+$ and $w^-$ by reducing the difference equations to the
Volterra-type discrete integral equations. The unity of the corresponding coefficients (the first one in \eqref{der} and the third 
one in \eqref{del}) appears to be crucial.

Define the (non-symmetric) Green kernels by
\begin{equation}\label{green}
\begin{split}
G_r(n,m;z) &:=\left\{
  \begin{array}{ll}
    \frac{z^{m-n}-z^{n-m}}{z-z^{-1}}, & \hbox{$m\ge n,$} \\
    0, & \hbox{$m\le n,$}
  \end{array}
\right. \\
G_l(n,m;z) &:=\left\{
  \begin{array}{ll}
   0, & \hbox{$m\ge n,$} \\
    \frac{z^{n-m}-z^{m-n}}{z-z^{-1}}, & \hbox{$m\le n,$}
  \end{array}
\right. \quad n,m\in\bz, \ \ z\in\bd_0.
\end{split}
\end{equation}
The basic property of the kernels can be verified directly
\begin{equation}\label{grker}
\begin{split}
G_{r,l}(n,m-1;z)+G_{r,l}(n,m+1;z)-\Bigl(z+\frac1z\Bigr)\,G_{r,l}(n,m;z) &=\d_{n,m}, \\
G_{r,l}(n-1,m;z)+G_{r,l}(n+1,m;z)-\Bigl(z+\frac1z\Bigr)\,G_{r,l}(n,m;z) &=\d_{n,m}.
\end{split}
\end{equation}
The kernels
\begin{equation}\label{traker}
\begin{split}
T_r(n,m;z) &:=-b_mG_r(n,m;z)+(1-a_{m-1}c_{m-1})G_r(n,m-1;z), \\
T_l(n,m;z) &:=-b_mG_l(n,m;z)+(1-a_{m}c_{m})G_l(n,m+1;z), \ \ z\in\bd_0,
\end{split}
\end{equation}
$n,m\in\bz$, are the key players of the game.

\begin{theorem}\label{difinteq}
The Jost solution $v^+=(v^+_k)$ of the difference equation $\eqref{der}$ at $+\infty$ satisfies the discrete Volterra equation
\begin{equation}\label{voltr}
v^+_n(z)=z^n+\sum_{m=n+1}^\infty T_r(n,m;z)v^+_m(z), \quad n\in\bz, \quad z\in\bd_0.
\end{equation}
Conversely, each solution $v=(v_n)$ of $\eqref{voltr}$ solves $\eqref{der}$.

Similarly, the Jost solution $w^-=(w^-_k)$ of $\eqref{del}$ at $-\infty$ satisfies the discrete Volterra equation
\begin{equation}\label{voltl}
w^-_n(z)=z^{-n}+\sum_{m=-\infty}^{n-1} T_l(n,m;z)w^-_m(z), \quad n\in\bz, \quad z\in\bd_0.
\end{equation}
Conversely, each solution $w=(w_n)$ of $\eqref{voltl}$ solves $\eqref{del}$.
\end{theorem}
\begin{proof}
We multiply the first relation \eqref{grker} for $G_r$ by $v^+_m$, 
\eqref{der} by $G_r(n,m)$, and subtract the later from the former
\begin{equation*}
\begin{split}
\Bigl[G_r(n,m+1)v^+_m-G_r(n,m)v^+_{m-1}\Bigr] &+\Bigl[-b_mG_r(n,m)+G_r(n,m-1)\Bigr]v^+_m \\
&-a_mc_mG_r(n,m)v^+_{m+1}=\d_{n,m}v^+_m. 
\end{split}
\end{equation*}
Next, taking into account that $G_r(n,n+1)=1$, $G_r(n,n)=0$, we sum up over $m$ from $n+1$ to $N$ 
\begin{equation*}
\begin{split}
G_r(n,N+1)v^+_N &+\sum_{m=n+1}^N \Bigl[-b_mG_r(n,m)+G_r(n,m-1)\Bigr]v^+_m \\
&-\sum_{m=n}^N a_mc_mG_r(n,m)v^+_{m+1}=v^+_n,
\end{split}
\end{equation*}
or
\begin{equation*}
v^+_n=G_r(n,N+1)v^+_N-a_Nc_NG_r(n,N)v^+_{N+1}+\sum_{m=n+1}^N T_r(n,m)v^+_m.
\end{equation*}
The latter equality holds for an arbitrary solution of \eqref{der}. If $v^+$ is the Jost solution at $+\infty$, then, by \eqref{green},
$$ \lim_{N\to\infty} G_r(n,N+1)v^+_N-a_Nc_NG_r(n,N)v^+_{N+1}=z^n, $$
and \eqref{voltr} follows.

The direct reasoning for \eqref{del} is the same. We multiply the first relation \eqref{grker} for $G_l$ by $w^-_m$, 
\eqref{del} by $G_l(n,m)$, and subtract the later from the former
\begin{equation*}
\begin{split}
\Bigl[G_l(n,m-1)w^-_m-G_l(n,m)w^-_{m-1}\Bigr] &+\Bigl[-b_mG_l(n,m)+G_r(n,m+1)\Bigr]w^-_m \\
&-a_{m-1}c_{m-1}G_l(n,m)w^-_{m-1}=\d_{n,m}w^-_m. 
\end{split}
\end{equation*}
The summation over $m$ from $-N$ to $n-1$ gives, as above
\begin{equation*}
w^-_n=G_l(n,-N-1)w^-_{-N}-a_{-N-1}c_{-N-1}G_l(n,-N)w^-_{-N-1}+\sum_{m=-N}^{n-1} T_l(n,m)w^-_m.
\end{equation*}
If $w^-$ is the Jost solution of \eqref{del} at $-\infty$, then
$$ \lim_{N\to\infty} G_l(n,-N-1)w^-_{-N}-a_{-N-1}c_{-N-1}G_l(n,-N)w^-_{-N-1}=z^{-n}, $$
and \eqref{voltl} follows.

To prove the converse statements, let $v=(v_n)$ be a solution of \eqref{voltr}. Then
\begin{equation*}
\begin{split}
v_{n-1}+v_{n+1} &=\Bigl(z+\frac1z\Bigr)\,z^n+T_r(n-1,n)v_n+T_r(n-1,n+1)v_{n+1} \\
&+\sum_{m=n+2}^\infty\Bigl[T_r(n-1,m)+T_r(n+1,m)\Bigr]v_m.
\end{split}
\end{equation*}
But
\begin{equation*}
\begin{split}
T_r(n-1,n)v_n &=-b_nv_n, \\
T_r(n-1,n+1)v_{n+1} &=\Bigl[-b_{n+1}G_r(n-1,n+1)+(1-a_nc_n)G_r(n-1,n)\Bigr]v_{n+1} \\
&=-\Bigl(z+\frac1z\Bigr)b_{n+1}v_{n+1}+(1-a_nc_n)v_{n+1} \\
&=\Bigl(z+\frac1z\Bigr)T_r(n,n+1)v_{n+1}+(1-a_nc_n)v_{n+1}, \\
T_r(n-1,m)+T_r(n+1,m) &=\Bigl(z+\frac1z\Bigr)T_r(n,n+1).
\end{split}
\end{equation*}
Finally,
\begin{equation*} 
v_{n-1}+v_{n+1}=-b_nv_n+(1-a_nc_n)v_{n+1}+\Bigl(z+\frac1z\Bigr)\,\Bigl(z^n+\sum_{m=n+1}^\infty T_r(n,m)v_m\Bigr),
\end{equation*}
which is \eqref{der}. The proof for the second converse statement is identical.
\end{proof}

It is convenient and instructive to introduce new variables in both \eqref{voltr} and \eqref{voltl}
\begin{equation}
\begin{split}
f_m^r &:=v_m^+z^{-m}-1, \quad \widetilde T_r(n,m;z):=T_r(n,m;z)z^{m-n}, \\
f_m^l &:=w_m^-z^m-1,    \quad \ \widetilde T_l(n,m;z):=T_l(n,m;z)z^{n-m},
\end{split}
\end{equation}
so the Volterra equations turn into
\begin{equation}\label{modvolr}
\begin{split}
f_n^r(z) &=g_n^r(z)+\sum_{m=n+1}^\infty \widetilde T_r(n,m;z)f_m^r(z), \\ 
g_n^r(z) &:=\sum_{m=n+1}^\infty \widetilde T_r(n,m;z),
\end{split}
\end{equation}
and
\begin{equation}\label{modvoll}
\begin{split}
f_n^l(z) &=g_n^l(z)+\sum_{m=-\infty}^{n-1} \widetilde T_l(n,m;z)f_m^l(z), \\ 
g_n^l(z) &:=\sum_{m=-\infty}^{n-1} \widetilde T_l(n,m;z).
\end{split}
\end{equation}
These are better than the original ones owing to the simple analytic properties of the kernels $\widetilde T_{r,l}$. Indeed,
it is not hard to verify that $\widetilde T_{r,l}(n,m;\cdot)$ are polynomials of $z$, and 
\begin{equation}\label{boundkerr}
\begin{split}
\bigl|\widetilde T_r(n,m;z)\bigr| &\le \d_m^r \min\Bigl\{(m-n)_+, \ \frac{2|z|}{|z^2-1|}\Bigr\}, \\ 
\d_m^r &:=|b_m|+|1-a_{m-1}c_{m-1}|, \quad n,m\in\bz, \quad z\in\ovl{\bd},
\end{split}
\end{equation}
\begin{equation}\label{boundkerl}
\begin{split}
\bigl|\widetilde T_l(n,m;z)\bigr| &\le \d_m^l \min\Bigl\{(n-m)_+, \ \frac{2|z|}{|z^2-1|}\Bigr\}, \\ 
\d_m^l &:=|b_m|+|1-a_{m}c_{m}|, \quad n,m\in\bz, \quad z\in\ovl{\bd}.
\end{split}
\end{equation}
In particular,
\begin{equation}\label{boundker}
\bigl|\widetilde T_{r,l}(n,m;z)\bigr|\le \d_m^{r,l}\,|\o(z)|, \quad \o(z):=\frac{2z}{1-z^2}\,, \quad z\in\bd_1:=\ovl\bd\bsl\{\pm1\}.
\end{equation}
So, the series for $g_n^{r,l}$ converge absolutely and uniformly on each compact subset of $\ovl\bd$, which omits $\pm1$, and
\begin{equation*}
|g_n^{r,l}(z)|\le |\o(z)|\dd_n^{r,l}, \quad \dd_n^r:=\sum_{m=n+1}^\infty \d_m^r, \quad \dd_n^l:=\sum_{m=-\infty}^{n-1} \d_m^l.
\end{equation*}

According to the general result \cite[Lemma 7.8]{te00} concerning the discrete Volterra equations, we have for $n\in\bz$ and $z\in\bd_1$
\begin{equation}\label{voltsol}
\begin{split}
|f_n^r(z)| =|z^{-n}v_n^+ -1| &\le |\o(z)|\dd_n^r\,\exp\bigl\{|\o(z)|\dd_n^r\bigr\}, \\ 
|f_n^l(z)| =|z^{n}w_n^- -1| &\le |\o(z)|\dd_n^l\,\exp\bigl\{|\o(z)|\dd_n^l\bigr\},
\end{split}
\end{equation}
or
\begin{equation}\label{voltsol1}
\begin{split}
|v_n^+ -z^n| &\le |z|^n|\o(z)|\dd_n^r\,\exp\bigl\{|\o(z)|\dd_n^r\bigr\}, \\ 
|w_n^- -z^{-n}| &\le |z|^{-n}|\o(z)|\dd_n^l\,\exp\bigl\{|\o(z)|\dd_n^l\bigr\}.
\end{split}
\end{equation}

\section{The Wronskian and the Lieb--Thirring inequality}

Let us go back to the main equation \eqref{3term}. Given its two solutions $u'=(u_n')$ and $u''=(u_n'')$, the equality below is obvious
$$ a_{n-1}(u_{n-1}'u_n''-u_n'u_{n-1}'')=c_n(u_{n}'u_{n+1}''-u_{n+1}'u_{n}''). $$
The Wronskian $W(u',u'')$ is naturally defined as
$$ W(u',u''):=\b_n(u_{n}'u_{n+1}''-u_{n+1}'u_{n}''), \quad \b_n:=a_n\prod_{j=-\infty}^n \frac{c_j}{a_j}\,. $$
Such choice of $\b_n$ makes the Wronskian independent of $n$.

From now on we put $u'=u^+$, $u''=u^-$. By the transition formulas \eqref{trans}, we can express $W(u^+,u^-)$ in terms of $v^+$ and $w^-$:
\begin{equation*}
\begin{split}
W(u^+,u^-) &=\b_n(u_n^+u_{n+1}^- -u_{n+1}^+u_n^-) \\
&=\b_n\Bigl(\prod_{j=n}^\infty a_j^{-1}\prod_{j=-\infty}^n c_j^{-1}\,v_n^+w_{n+1}^- 
-\prod_{j=n+1}^\infty a_j^{-1}\prod_{j=-\infty}^{n-1} c_j^{-1}\,v_{n+1}^+w_n^-\Bigr) \\
&=\prod_{j=-\infty}^\infty a_j^{-1}\,(v_n^+w_{n+1}^- - a_nc_n\,v_{n+1}^+w_n^-).
\end{split}
\end{equation*}
So the bound for the Wronskian will follows from the inequalities \eqref{voltsol1}. Note also that
\begin{equation*}
\dd^{r,l}_n\le\dd:=\sum_{j=-\infty}^\infty (|b_j|+|1-a_jc_j|), \qquad n\in\bz.
\end{equation*}

We are now ready for

{\it Proof of Theorem \ref{mainth}}. Put
\begin{equation}
\begin{split}
U(z) &:=\frac{\o(z)}2 \prod_{j=-\infty}^\infty a_j W(u^+,u^-) \\
&=\frac{\o(z)}2 \bigl(v_0^+(z)w_1^-(z)-v_1^+(z)w_0^-(z)+(1-a_0c_0)v_1^+(z)w_0^-(z)\bigr).
\end{split}
\end{equation}
If
$$ p_j(z):=v_j^+(z)-z^j, \quad q_j(z):=w_j^-(z)-z^{-j}, \quad j=0,1, $$
then
\begin{equation*}
\begin{split}
U(z) &:=\frac{\o(z)}2 \Bigl[(1+p_0(z))(z^{-1}+q_1(z))-(z+p_1(z))(1+q_0(z))+(1-a_0c_0)v_1^+(z)w_0^-(z)\Bigr] \\
&=1+\frac{\o(z)}2 d(z), \quad d(z)=d_1(z)-d_2(z)+d_3(z),
\end{split}
\end{equation*} 
where
\begin{equation*}
\begin{split}
d_1(z) &:=q_1(z)+z^{-1}p_0(z)+p_0(z)q_1(z), \\
d_2(z) &:=p_1(z)+zq_0(z)+p_1(z)q_0(z), \\
d_3(z) &:=(1-a_0c_0)v_1^+(z)w_0^-(z).
\end{split}
\end{equation*}
We proceed with the upper bound for the function $U$ term by term.

1. For $d_1$ we have
$$ \frac{|\o(z)|}2\,|d_1(z)|\le \frac{|\o(z)|}2\bigl(|q_1(z)|+|z^{-1}p_0(z)|+|p_0(z)q_1(z)|\bigr). $$
In view of \eqref{voltsol1}
\begin{equation*}
\begin{split}
\frac{|\o(z)|}2\bigl(|q_1(z)|+|z^{-1}p_0(z)|\bigr) &\le\frac{|\o(z)|^2}{|z|}\dd e^{|\o(z)|\dd}, \\
\frac{|\o(z)|}2 |p_0(z)q_1(z)| &\le \frac{|\o(z)|^3}{2|z|}\dd^2 e^{2|\o(z)|\dd}\le \frac{|\o(z)|^2}{2|z|}\dd e^{3|\o(z)|\dd}, 
\end{split}
\end{equation*}
and so
$$ \frac{|\o(z)|}2\,|d_1(z)|\le\frac{3|\o(z)|^2}{2|z|}\dd e^{3|\o(z)|\dd}. $$
Next, it is clear that
$$
|1-z^2|+|z|\ge 1, \quad 1+\frac{|\o(z)|}2\ge\frac{|\o(z)|}{2|z|} $$
or
$$ \left|\frac{\o(z)}{z}\right|\le 2(1+|\o(z)|). $$
Hence,
$$ \frac{3|\o(z)|^2}{2|z|}\le 3|\o(z)|(1+|\o(z)|) $$
and finally
\begin{equation}\label{d1}
\begin{split}
&{}\frac{|\o(z)|}2\,|d_1(z)| \le 3(|\o(z)|+|\o(z)|^2)\dd e^{3|\o(z)|\dd} \\
&\le 3\bigl\{|\o(z)|(\dd^{1/2}+\dd)+|\o(z)|^2(\dd^{1/2}+\dd)^2\bigr\}\,e^{3|\o(z)|(\dd^{1/2}+\dd)}.
\end{split}
\end{equation}

2. For $d_2$ we have
$$ \frac{|\o(z)|}2\,|d_2(z)|\le \frac{|\o(z)|}2\bigl(|p_1(z)|+|zq_0(z)|+|p_1(z)q_0(z)|\bigr). $$
It is immediate from \eqref{voltsol1} that
\begin{equation*}
\begin{split} 
|p_1(z)| &\le |\o(z)|\dd e^{|\o(z)\dd}, \quad |q_0(z)|\le |\o(z)|\dd e^{|\o(z)\dd}, \\ 
|p_1(z)q_0(z)| &\le |\o(z)|^2\dd^2 e^{2|\o(z)\dd},
\end{split}
\end{equation*}
and so
\begin{equation}\label{d2}
\begin{split}
&{}\frac{|\o(z)|}2\,|d_2(z)| \le |\o(z)|^2\dd e^{|\o(z)|\dd}+\frac{|\o(z)|^3}2\dd^2 e^{2|\o(z)|\dd} \\
&\le 2|\o(z)|^2\dd e^{3|\o(z)|\dd}\le 2|\o(z)|^2(\dd^{1/2}+\dd)^2\,e^{3|\o(z)|(\dd^{1/2}+\dd)}.
\end{split}
\end{equation}

3. For $d_3$ we have by \eqref{voltsol1},
$$ \frac{|\o(z)|}2\,|d_3(z)|\le \frac{|\o(z)|}2 \dd\Bigl(1+|\o(z)|\dd e^{|\o(z)|\dd}\Bigr)^2, $$
and since $1+xe^x\le e^{2x}$ for $x\ge0$, then
\begin{equation}\label{d3}
\frac{|\o(z)|}2\,|d_3(z)|\le \frac{|\o(z)|}2 \dd e^{4|\o(z)|\dd}\le \frac{|\o(z)|}2 (\dd^{1/2}+\dd) e^{4|\o(z)|(\dd^{1/2}+\dd)}.
\end{equation}

A combination of \eqref{d1} -- \eqref{d3} produces the following bound for $U$
\begin{equation}\label{perdetto1}
\begin{split}
|U(z)-1| &\le (4x+5x^2)e^{4x}, \qquad x:=|\o(z)|(\dd^{1/2}+\dd), \\
|U(z)| &\le (1+4x+5x^2)e^{4x}\le e^{8x}.
\end{split}
\end{equation}

By the non-self-adjoint version of \cite[Proposition 10.6]{kisi03} (the calculation there is algebraic and so immediately extends 
to the non-self-adjoint case), $U(\cdot)=L(\cdot,J)$,
so we come to the bound for the perturbation determinant
\begin{equation}\label{bopd}
\log |L(z,J)|\le \frac{16|z|}{|1-z^2|}\,(\dd^{1/2}+\dd), \qquad L(0,J)=1.
\end{equation} 

The rest is standard nowadays. According to \cite[Theorem 4]{haka11}, for each $\ep\in(0,1)$ there is a constant $C(\ep)>0$ so that the
Blaschke-type condition holds for the zero set (divisor) $Z(L)$
$$ \sum_{\z\in Z(L)} (1-|\z|)\frac{|\z^2-1|^\ep}{|\z|^\ep}\le C(\ep)(\dd^{1/2}+\dd), $$
(each zero is taken with its multiplicity). The latter inequality turns into \eqref{mainlt} when we go over to the Zhukovsky images
and take into account the distortion for the Zhukovsky function \cite[Lemma 7]{haka11}. The proof is complete.

\smallskip

For the discrete Schr\"odinger operators $J$ $(a_j=c_j\equiv 1)$ \eqref{mainlt1} follows from 
$$ \dd=\sum_{j=-\infty}^\infty |b_j|=\|J-J_0\|_1. $$ 

\begin{remark}\label{speenc}
As a byproduct, the first bound in \eqref{perdetto1} for the perturbation determinant provides some information on the location of 
the discrete spectrum (spectral enclosure). Indeed, let $\ka$ be a unique positive root of the equation
$$ (4x+5x^2)e^{4x}=1, \qquad \ka\approx 0.129. $$
Then $L\not=0$ in $\bd$ as long as 
\begin{equation*} 
|\o(z)|(\dd^{1/2}+\dd)<\ka, \qquad \frac{|z|}{|1-z^2|}<\frac{\ka}{2(\dd^{1/2}+\dd)}\,, 
\end{equation*}
or in terms of the Zhukovsky images
\begin{equation}\label{spenc1}
\s_d(J)\subset \left\{\l\in\bc\backslash [-2,2]: \ |\l^2-4|\le \left(\frac{2(\dd^{1/2}+\dd)}{\ka}\right)^2\right\}. 
\end{equation}
So, the discrete spectrum lies in a certain Cassini oval.

The spectral enclosure is normally derived from the Birman--Schwinger principle. Precisely,
$$ \l(z)\in\s_d(J) \ \Rightarrow \ \|K(z)\|\le1, $$
$K$ is the Birman--Schwinger operator. In our case one has
\begin{equation}\label{spenc2}
\s_d(J)\subset \bigl\{\l\in\bc\backslash [-2,2]: \ |\l^2-4|\le 324\|J-J_0\|_1^2\bigr\}. 
\end{equation}
It might be curious comparing the ovals in \eqref{spenc1} and \eqref{spenc2}.

For the discrete Schr\"odinger operators the sharp oval which contains the discrete spectrum is known \cite{ibst19}
\begin{equation}\label{spenc3}
\s_d(J)\subset \bigl\{\l\in\bc\backslash [-2,2]: \ |\l^2-4|\le \|J-J_0\|_1^2\bigr\}. 
\end{equation}

\end{remark}

\end{document}